\documentclass[11pt]{amsart}
\usepackage[margin=1in]{geometry}

\usepackage{bbm, physics, amssymb, hyperref}
\hypersetup{
    colorlinks=true,
    citecolor=blue
}

\newcommand{\subeq}{\subseteq}
\newcommand{\ub}{\underbrace}
\newcommand{\wh}{\widehat}

\DeclareMathOperator{\vol}{vol}
\DeclareMathOperator{\Gal}{Gal}

\newcommand{\eps}{\varepsilon}
\newcommand{\lam}{\lambda}
\newcommand{\Lam}{\Lambda}
\newcommand{\vhi}{\varphi}

\newcommand{\QQ}{\mathbb Q}
\newcommand{\RR}{\mathbb R}
\newcommand{\TT}{\mathbb T}
\newcommand{\ZZ}{\mathbb Z}
\newcommand{\one}{\mathbf{1}}

\newcommand{\cC}{\mathcal C}

\newcommand{\cI}{\mathcal I}

\newcommand{\cO}{\mathcal O}
\newcommand{\cP}{\mathcal P}

\newcommand{\fa}{\mathfrak a}

\newcommand{\ff}{\mathfrak f}

\newcommand{\fp}{\mathfrak p}

\newcommand{\bigp}[1]{\left( #1 \right)} 
\newcommand{\bigb}[1]{\left[ #1 \right]} 
\newcommand{\bigc}[1]{\left\{ #1 \right\}} 
\newcommand{\biga}[1]{\left\langle #1 \right\rangle} 
\newcommand{\ceil}[1]{\left\lceil #1 \right\rceil}
\renewcommand{\norm}[1]{\| #1 \|}

\newcommand{\p}[1]{\partial #1}

\newcommand{\textif}{\quad\text{if }}

\newtheorem{thm}{Theorem}[section]
\newtheorem{theorem}[thm]{Theorem}

\newtheorem{lemma}[thm]{Lemma}

\title{A Bombieri-Vinogradov Theorem for primes in short intervals and small sectors}

\author{Tanmay Khale}
\address[Tanmay Khale]{Department of Mathematics, Lafayette College}
\email{khalet@lafayette.edu}

\author{Cooper O'Kuhn}
\address[Cooper O'Kuhn]{Department of Mathematics, University of Florida}
\email{cokuhn@ufl.edu}

\author{Apoorva Panidapu}
\address[Apoorva Panidapu]{Department of Mathematics, San Jose State University}
\email{apoorva002@gmail.com}

\author{Alec Sun}
\address[Alec Sun]{Department of Mathematics, Harvard University}
\email{sundogx@gmail.com}

\author{Shengtong Zhang}
\address[Shengtong Zhang]{Department of Mathematics, Massachusetts Institute of Technology}
\email{stzh1555@mit.edu}

\begin{document}

\begin{abstract}
   Let $K$ be a finite Galois extension of $\QQ$. We count primes in short intervals represented by the norm of a prime ideal of $K$ satisfying a small sector condition determined by Hecke characters. We also show that such primes are well-distributed in arithmetic progressions in the sense of Bombieri-Vinogradov. This extends previous work of Duke and Coleman.
\end{abstract}

\maketitle

\section{Introduction}\label{introduction}
A famous open problem is whether or not there exist infinitely many primes of the form $p = n^2 + 1$. A natural approximation to this problem is to count primes of the form $p = a^2 + b^2$ with $\abs{a} < p^{\frac12 - \delta}$ for some $\delta>0$. The set of such primes has zero density in the set of all primes. Counting these primes is equivalent to counting Gaussian primes $a+bi$ that lie in a small sector of the complex plane. Kubilius \cite{kubilius} proved the existence of a constant $\delta_0>0$ such that for all $0<\delta<\delta_0$, we have 
$$\#\bigc{p \text{ prime} \colon p\le x, p = a^2 + b^2, \abs{a} < p^{\frac{1}{2} - \delta}} \sim \frac{cx^{1 - \delta}}{\log{x}}$$
for some constant $c$. By partial summation, we have the equivalent asymptotic
\[
\sum_{\substack{p\leq x \\ p = a^2+b^2 \\ |a|<p^{\frac{1}{2}-\delta}}}\log p \sim cx^{1-\delta}.
\]
The best $\delta_0$ that has been obtained to date is $\frac{12}{37}$ by Maknys \cite{maknys}. 

More generally, one can consider primes represented by norm forms for an imaginary quadratic extension $K = \QQ(\sqrt{-m})$. For example, when $m\equiv 1,2\bmod 4$ is square-free, a prime $p$ is of the form $x^2 + my^2$ if and only if $p = N\fp$ for some principal ideal $\fp\subeq \cO_K$. Fix a nonzero ideal $\ff\subeq \cO_K$, and let $g$ be the number of units $\eps\equiv 1\bmod \ff$ in $\QQ(\sqrt{-m})$. Let $\lam$ denote a generator for the infinite-order Hecke characters mod $\ff$ for an imaginary quadratic extension $\QQ(\sqrt{-m})$ such that $$\lam((\alpha)) = \bigp{\frac{\alpha}{\abs{\alpha}}}^{g}$$ for principal ideals $(\alpha)$ with $\alpha\equiv 1 \bmod \ff$. Denote by $\cI_\ff$ the ideal class group mod $\ff$, and consider the distribution of prime ideals within a particular ideal class mod $\ff$, which we denote by $I$.


Refining the result of Kubilius, Coleman \cite{coleman-quadratic} showed that for imaginary quadratic extensions $K = \QQ(\sqrt{-m})$, prime ideals with argument in a specified range exhibit regularity in short intervals $[x, x+h)$ with $h = x^{1-\delta'}$ for some $\delta'>0$. To be precise, Coleman \cite{coleman-quadratic} proved that for small $\eps>0$, we have $$\sum_{\substack{\fp\in I \\ N\fp = p \text{ prime}\\ x\le p < x+h \\ \phi_1 \le \arg \lam(\fp) \le \phi_2}} \log p \sim \frac{(\phi_2 - \phi_1)h}{2\pi\abs{\cI_\ff}}$$ 
for $0 \le \phi_1 \le \phi_2 \le 2\pi, \phi_2 - \phi_1 > x^{-\frac{5}{24} + \eps}$, and $x^{\frac{19}{24} + \eps} \le h \le x$.

In the case of an imaginary quadratic extension, a prime ideal $\fp\subeq \cO_K$ lying over an unramified prime $p$ satisfies the conditions in Coleman's result \cite{coleman-quadratic} except the one related to the argument if and only if its conjugate prime ideal $\bar{\fp}$, does, and we have $\arg \lam(\fp) = -\arg \lam(\bar{\fp})$. Hence Coleman's result \cite{coleman-quadratic} also produces an asymptotic for the count of rational primes represented by norms of these prime ideals, with the caveat that for intervals $(\phi_1, \phi_2)$ symmetric about $\pi$, each rational prime is counted twice.

Coleman claimed a generalization of this prime counting result for all number fields $K$ in the proof of \cite[Theorem 2]{coleman-general} using a purported one-to-one correspondence between rational primes and prime ideals. However, we observe from a close inspection of \cite{coleman-general} that this claim is not sufficiently justified because the number of prime ideals $\fp$ lying over a prime $p$ that satisfy a certain argument condition may not be constant.

In this paper, we extend Coleman's count of primes represented by norms of prime ideals in the case of imaginary quadratic extensions. We produce an asymptotic count of rational primes in the more general setting of Galois extensions $K/\QQ$. To state our results, we first require some definitions of Hecke characters.

\subsection{Hecke characters}

Let $K$ be a number field of degree $n = r_1 + 2r_2$ with real embeddings $\sigma_1,\ldots,\sigma_{r_1 }$ and complex embeddings $\sigma_{r_1 + 1},\ldots,\sigma_{r_1 + r_2}$, where we denote the image of $\alpha\in K$ by $\sigma_j(\alpha) = \alpha^{(j)}$. Let $\ff\subeq \cO_K$ be a fixed nonzero ideal. Let $I_\ff$ denote the group of fractional ideals relatively prime to $\ff$, and define a group of principal ideals $$P_\ff = \bigc{ (\alpha)\in I_\ff \colon \alpha\in K^\times, \alpha\equiv 1 \bmod\ff, \alpha\succ 0}.$$ Here the notation $\alpha\succ 0$ means that $\alpha$ is totally positive. An infinite-order Hecke character $\lam$ is a character on $I_\ff$ with the property that there exist suitable $v_j\in \RR$ and $u_l\in \ZZ$ such that on principal ideals $(\alpha)\in P_\ff$, we have $$\lam((\alpha)) = \prod_{j=1}^{r_1+r_2} \big|\alpha^{(j)}\big|^{iv_j} \prod_{l=r_1+1}^{r_1 + r_2} \bigp{\frac{\alpha^{(l)}}{\abs{\alpha^{(l)}}}}^{u_l}.$$ For this to be well-defined, we need $\lam(\eps) = 1$ for all units $\eps\equiv 1 \bmod \ff$ satisfying $\eps\succ 0$. We also require $\lam(\alpha) = 1$ for all $\alpha\in \QQ$, implying the condition $$\sum_{j=1}^{r_1 + r_2}v_j = 0.$$ The group of characters $\lam$ has a multiplicative basis of $d = n-1$ elements. Fixing such a multiplicative basis $\lam_1,\ldots,\lam_{d}$, each Hecke character mod $\ff$ can be written as $$\mu\lam^\mathbf{m}(\fa) = \mu(\fa) \prod_{j=1}^{d} \lam_j^{m_j}(\fa),\quad \fa\in I_\ff,$$
for some $\mathbf{m} = (m_1,\ldots,m_d)\in \ZZ^{d}$ and for some character $\mu$ on $I_\ff/P_\ff$, which we refer to as a narrow class character mod $\ff$. Define the argument $\vec{\phi}(\fa) = (\phi_1(\fa),\ldots,\phi_d(\fa))\in \RR^{d}/\ZZ^d = \TT^d$ for $\fa\in I_\ff$ by $\lam_j(\fa) = e^{2\pi i \phi_j(\fa)}$. Let $\norm{\cdot}$ denote the sup-norm on $\TT^d$.

\subsection{Main results}

We now state the two main results of this paper. The first main theorem counts primes $p$ in short intervals represented by norms of prime ideals with argument lying in a narrow sector of the form $\norm{\vec{\phi}(\fp) - \vec{\phi}_0} < p^{-\delta}$. We also impose the condition that these prime ideals lie in a fixed narrow ideal class, which is motivated by counting primes represented by norm forms. Let $\fa\in I$ be an ideal with a special basis $\{\alpha_j\}$ satisfying the following conditions in \cite[Section 3.2]{duke}:
\begin{itemize}
    \item $\det\bigp{\alpha_j^{(i)}} = \sqrt{\Delta_{K}} \cdot N\fa$, where $\alpha_j^{(i)}$ denotes the matrix whose $(i,j)$-th entry is $\alpha_j^{(i)}$ and $\Delta_{K}$ denotes the discriminant of $K$.
    \item $\alpha_m\succ 0$ for some $m$.
\end{itemize}
A prime $p$ can be represented as $$p = N\bigp{\sum x_j \alpha_j} N(\fa)^{-1},\quad \sum x_j \alpha_j\succ 0$$ if and only if $p = N\fp$ for some prime ideal $\fp\in I^{-1}$. Duke \cite[Theorem 3.2]{duke} notes that counting primes in a given narrow ideal class whose associated prime ideals lie in small sectors is related to counting primes represented by norm forms of a number field $K$ with all but one component small via the geometric interpretation of $\vec{\phi}$. The first main theorem of this paper is:

\begin{theorem}\label{theorem:galois-primes}
Let $K/\QQ$ be a finite Galois extension of degree $n$. Let $0<\delta,\delta'< \frac{2}{5n}$ and $\vec{\phi}_0 \in \TT^{n-1}$ be fixed. Denote by $\cI$ the narrow ideal class group mod $(1)$, and let $I\in \cI$ be a fixed narrow ideal class mod $(1)$. Define
\[
\mathcal{P}_{I,\delta}(\vec{\phi}_0)=\bigc{p \in \cP \colon \text{$\exists \mathfrak{p}\in I$ with $N\mathfrak{p}=p$ and $\norm{\vec{\phi}(\mathfrak{p})-\vec{\phi}_0}<p^{-\delta}$}}.
\]
Then there are constants $c$ and $c_1>0$ depending on $\delta$ and $\delta'$ such that for any constant $0<c_2<\frac13$ we have $$\sum_{\substack{ x\le p < x+h \\ p\in\mathcal{P}_{I,\delta}(\vec{\phi}_0)}} \log p  = \frac{cx^{-(n-1)\delta} h}{\abs{\cI}} \bigp{1 + O\bigp{e^{-c_1 (\log x)^{c_2}}}},\quad h = x^{1-\delta'}.$$ 
\end{theorem}

The second main theorem shows that the primes counted in Theorem \ref{theorem:galois-primes} are well-distributed in the sense of Bombieri-Vinogradov. The classical Bombieri-Vinogradov Theorem \cite{davenport} states that for any $0<\theta<\frac12$ and $A > 0$, we have $$\sum_{q \leq Q} \max_{(a,q) = 1} \Bigg|\sum_{\substack{p\leq x \\ p\equiv a\bmod{q}}} \log p - \frac{x}{\phi(q)}\Bigg|\ll_A \frac{x}{(\log x)^A},\quad Q = x^{\theta}.$$ A short interval generalization of the Bombieri-Vinogradov Theorem has been established by Jutila \cite{jutila} using ideas by Huxley and Jutila \cite{huxley-jutila}.

For imaginary quadratic extensions $K=\QQ(\sqrt{-m})$, Coleman and Swallow \cite{coleman-swallow} have proven a Bombieri-Vinogradov theorem for prime ideals with norm in short intervals satisfying the conditions in Theorem \ref{theorem:galois-primes}. We prove an analogous result for rational primes represented by norms of such ideals that in addition holds for any Galois extension $K/\mathbb{Q}$.

\begin{theorem} \label{theorem:bombieri-vinogradov-galois-primes}
Let $K/\QQ$ be a finite Galois extension of degree $n$. Let $\delta,\delta',\theta>0$ be constants with $2\theta + \max\bigp{\delta,\delta'} < \frac{2}{5n}$, and let $\vec{\phi}_0 \in \TT^{n-1}$ be fixed. Denote by $\cI$ the narrow ideal class group mod $(1)$, and let $I\in \cI$ be a fixed narrow ideal class mod $(1)$. Define
\[
\mathcal{P}_{I,\delta}(\vec{\phi}_0)=\bigc{p \text{ prime}\colon \text{$\exists \mathfrak{p}\in I$ with $N\mathfrak{p}=p$ and $\norm{\vec{\phi}(\mathfrak{p})-\vec{\phi}_0}<p^{-\delta}$}}.
\]
Then there is a constant $c$ such that for $A>0$ we have $$\sum_{\substack{q \leq Q \\ H_K^+\cap \QQ(\zeta_q) = \QQ }}\max_{\gcd(a, q) = 1}\Bigg|\sum_{\substack{ x\le p < x+h \\ p \equiv a\bmod q\\ p\in\mathcal{P}_{I,\delta}(\vec{\phi}_0)}} \log p - \frac{chx^{-(n-1)\delta}}{\vhi(q) \abs{\cI}}\Bigg| \ll_{A} \frac{hx^{-(n-1)\delta}}{(\log x)^A}, \quad Q = x^{\theta}, \quad h = x^{1-\delta'}$$ where $H_K^+$ denotes the narrow Hilbert class field of $K$.
\end{theorem}

Using partial summation, both theorems can be modified to count the number of primes without a $\log p$ weighting satisfying the given Hecke character conditions. Although Theorem \ref{theorem:bombieri-vinogradov-galois-primes} holds only for Galois extensions $K/\QQ$, we remark that the proof of Theorem \ref{theorem:bombieri-vinogradov-galois-primes} starting from Section \ref{section:refined-explicit-formula} yields a Bombieri-Vinogradov inequality for all number fields $K$, not necessarily Galois over $\QQ$, at the cost of counting prime ideals that satisfy the given Hecke character conditions rather than rational primes. This extends the localized Bombieri-Vinogradov theorem for prime ideals in imaginary quadratic extensions proven by Coleman and Swallow in \cite{coleman-swallow}.

\subsection*{Acknowledgements}

This research was funded by the NSF (DMS-2002265), the NSA (H98230-20-1-0012), the Templeton World Charity Foundation, and the Thomas Jefferson Fund at the University of Virginia. The authors thank Ken Ono, Jesse Thorner and Wei-Lun Tsai for advising this research and for many helpful conversations, as well as Fabian Gundlach for pointing out a minor error. Finally, the authors are grateful to the anonymous reviewer for many suggestions that improved the presentation of this article.

\section*{Notation}
\begin{itemize}
    \item $p$ denotes a rational prime.
    \item $\fa$ denotes an ideal of $\cO_K$.
    \item $\fp$ denotes a prime ideal of $\cO_K$.
    \item $\cI$ denotes the narrow ideal class group of $K$.
    \item $H_K^+$ denotes the narrow Hilbert class field of $K$.
    \item $\Delta_K$ denotes the discriminant of $K$.
    \item $I\in \cI$ denotes a narrow ideal class.
    \item $\lambda^{\mathbf{m}}$ denotes an infinite-order Hecke character when $\mathbf{m}\neq \mathbf{0}$.
    \item $\mu$ denotes an arbitrary finite-order Hecke character.
    \item $\eta$ denotes a narrow ideal class character with conductor $(1)$.
    \item $\chi$ denotes a Dirichlet character.
    \item $\one$ denotes the indicator function.
    \item A sum over the variable $\fa$ denotes a sum over all ideals $\fa\subeq \cO_K$.
    \item A sum over the variable $\fp$ denotes a sum over all prime ideals $\fp\subeq \cO_K$.
    \item $\sum^*$ denotes a sum over only primitive characters.
    \item $\sum_q'$ denotes a sum over positive integers $q$ such that $H_K^+\cap \QQ(\zeta_q) = \QQ$.
    \item $\norm{\cdot}$ denotes the sup-norm.
    \item The implied constants in the notation $O(\cdot)$, $\ll$, and $\gg$ depend on the number field $K$.
\end{itemize}

\section{Proof of Theorem \ref{theorem:galois-primes}} \label{section:proof-galois-primes}

In this section, we prove that
\begin{equation} \label{equation:x-delta}
    \sum_{\substack{x\le p < x+h \\ \exists \fp \in I \text{ s.t.} \\ N\fp = p \\ \norm{\vec{\phi}(\fp) - \vec{\phi}_0}< x^{-\delta}}} \log p  = \frac{cx^{-(n-1)\delta} h}{\abs{\cI}} \bigp{1 + O\bigp{e^{-c_1 (\log x)^{c_2}}}}.
\end{equation}
for some constants $c$ and $c_1>0$ depending on $\delta$ and $\delta'$, and any constant $0<c_2<\frac{1}{3}$. Once we prove this result, we will show that we can replace the condition $\norm{\vec{\phi}(\mathfrak{p})-\vec{\phi}_0}<x^{-\delta}$ by $\norm{\vec{\phi}(\mathfrak{p})-\vec{\phi}_0}<p^{-\delta}$, thus proving Theorem \ref{theorem:galois-primes}.

To show \eqref{equation:x-delta}, we adapt the principle of inclusion-exclusion to express the count of rational primes in terms of counts of certain prime ideals associated with a set of automorphisms $\sigma_1,\ldots,\sigma_k\in \Gal(K/\QQ)$. We then show an explicit formula for the count of prime ideals and use a zero-density estimate to bound the sum over zeros of Hecke $L$-functions. Fix $d=n-1$ to be the dimension of the space of infinite-order Hecke characters, and write $\vec{\phi}_0 = \bigp{\phi_{01},\ldots,\phi_{0d}}$.

\subsection{Principle of inclusion-exclusion} \label{section:principle-inclusion-exclusion} 
We introduce a key lemma motivated by the principle of inclusion-exclusion that lets us reduce the problem of counting rational primes that satisfy the given conditions above to counting prime ideals.

\begin{lemma} \label{lemma:principle-inclusion-exclusion}
Suppose that $p$ does not ramify in $K$. Fix an arbitrary ordering on the $n$ elements of $\Gal(K/\QQ)$. Then we have $$\one\bigp{\substack{\exists \fp \in I \text{ s.t.} \\ N\fp = p \\ \norm{\vec{\phi}(\fp) - \vec{\phi}_0} < x^{-\delta}}} = \frac{1}{n} \sum_{k=1}^n (-1)^{k-1} \sum_{\substack{\sigma_1,\ldots, \sigma_k \in \Gal(K/\QQ) \\ \sigma_1<\cdots < \sigma_k}} ~ \sum_{\fp \text{ over }  p} ~ \one\bigp{\substack{\forall 1\le j\le k \\ N(\sigma_j(\fp)) = p \\ \sigma_j(\fp)\in I  \\ \norm{\vec{\phi}(\sigma_j(\fp)) - \vec{\phi}_0} < x^{-\delta}}},$$ where $\one(\cdot)$ denotes the indicator function of an event.
\end{lemma}

\begin{proof}
Since $K/\QQ$ is Galois, all primes lying over $p$ have the same inertia degree. If $p$ does not split completely in $K$, then all primes $\fp$ lying over $p$ have $N\fp\neq p$ and $N(\sigma_j(\fp))\neq p$ because the inertia degrees are greater than 1. Assume now that $p$ splits completely. Suppose that $m\ge 0$ primes $\fp$ lying over $p$ satisfy $\fp\in I, N\fp = p$, and $\norm{\vec{\phi}(\fp) - \vec{\phi}_0} < x^{-\delta}$. Then the left hand side is $\one_{m>0}$ by definition of $m$. To compute the right hand side, fix $1\le k\le n$. For each prime $\fp$ lying over $p$, there are $\binom{m}{k}$ tuples of $k$ automorphisms $\sigma_1<\cdots <\sigma_k$ such that for all $1\le j\le k$, we have $\sigma_j(\fp)\in I, N(\sigma_j(\fp)) = p$, and $\norm{\vec{\phi}(\sigma_j(\fp)) - \vec{\phi}_0} < x^{-\delta}$. We use the convention $\binom{m}{k} = 0$ for $m<k$. Since there are $n$ primes lying over $p$, we conclude that the right hand side is $$\frac{1}{n} \sum_{k=1}^n (-1)^{k-1} n\binom{m}{k} = \one_{m>0}$$ as desired.
\end{proof}

By Lemma \ref{lemma:principle-inclusion-exclusion}, it suffices to estimate
\begin{align*}
    \sum_{\substack{ x\le p < x+h \\ \exists \fp \in I \text{ s.t.} \\ N\fp = p \\ \norm{\vec{\phi}(\fp) - \vec{\phi}_0}< x^{-\delta}}} \log p
    &= \frac{1}{n}\sum_{ \sigma_1<\cdots<\sigma_k} (-1)^{k-1} \sum_{\substack{x\le N\fp < x+h \\ \forall 1\le j\le k \\ \sigma_j(\fp)\in I \\  N(\sigma_j(\fp)) = p \\ \norm{\vec{\phi}(\sigma_j(\fp)) - \vec{\phi}_0} < x^{-\delta}}} \log (N\fp).
\end{align*}
Since the number of $k$-tuples of automorphisms $\sigma_1<\cdots <\sigma_k$ is at most $2^n$, which is constant given the number field $K$, it suffices to produce an asymptotic, in the form of Theorem \ref{theorem:galois-primes}, for $$\sum_{\substack{x\le N\fp < x+h \\\forall 1\le j\le k \\ \sigma_j(\fp)\in I \\  N(\sigma_j(\fp)) = p \\ \norm{\vec{\phi}(\sigma_j(\fp)) - \vec{\phi}_0} < x^{-\delta}}} \log (N\fp)$$ given a fixed set of automorphisms $\sigma_1<\cdots <\sigma_k$.

Define $\Lam(\fa)$ to be the generalized von Mangoldt function $$\Lam(\fa) = \begin{cases} \log (N\fp) & \fa = \fp^k\text{ for some prime }\fp \\ 0 & \text{otherwise}.\end{cases}$$ Because there are $\ll x^{\frac12}$ ideals $\fa$ satisfying $N\fa = p^k$ with $k\ge 2$ in the interval $x\le N\fa < x+h$, we have
\begin{equation}\label{equation:von-mangoldt}
    \sum_{\substack{x\le N\fp < x+h \\ \forall 1\le j\le k \\ \sigma_j(\fp)\in I \\  N(\sigma_j(\fp)) = p \\ \norm{\vec{\phi}(\sigma_j(\fp)) - \vec{\phi}_0} < x^{-\delta}}} \log (N\fp) = O\bigp{x^{\frac{1}{2}}\log x} + \sum_{\substack{x\le N\fa < x+h \\ \forall 1\le j\le k \\ \sigma_j(\fa)\in I  \\ \norm{\vec{\phi}(\sigma_j(\fa)) - \vec{\phi}_0} < x^{-\delta}}} \Lam(\fa).
\end{equation}
The error term is negligible to the main term in Theorem \ref{theorem:galois-primes} for $0<\delta, \delta' < \frac{2}{5n}$ and hence can be ignored.

One important observation is that the condition $\sigma_j(\fa)\in I$ is equivalent to the condition $\fa\in \sigma_j^{-1}(I)$. Here we use the fact that if $\fa\succ 0$ then $\sigma(\fa) \succ 0$ for all $\sigma\in \Gal(K/\QQ)$, which implies that the Galois action is well-defined on the narrow ideal class group. The conditions $\fa\in \sigma_j^{-1}(I)$ for $1\le j\le k$ can be simultaneously satisfied if and only if the narrow ideal class $\sigma_j^{-1}(I)$ is constant for $1\le j\le k$. In the case that $\sigma_j^{-1}(I)$ is constant, we denote the narrow ideal class by $I'$. Hence we have
\begin{equation}\label{equation:von-mangoldt-ideal-class}
    \sum_{\substack{x\le N\fa < x+h \\ \forall 1\le j\le k \\ \sigma_j(\fa)\in I  \\ \norm{\vec{\phi}(\sigma_j(\fa)) - \vec{\phi}_0} < x^{-\delta}}} \Lam(\fa) =
    \begin{cases}
        \displaystyle\sum_{\substack{\fa\in I'\\ x\le N\fa < x+h \\ \forall 1\le j\le k \\ \norm{\vec{\phi}(\sigma_j(\fa)) - \vec{\phi}_0} < x^{-\delta}}} \Lam(\fa) & \quad\text{if }\sigma_j^{-1}(I) = I',\quad \forall 1\le j\le k \\ \\
        0 & \quad\text{otherwise}.
    \end{cases}
\end{equation}
We restrict to the non-trivial former case for the rest of the proof of Theorem \ref{theorem:galois-primes}. Another important observation from class field theory is that for each $j$ and each infinite-order Hecke character $\lam_i$, 
$$\lam_i \circ \sigma_j:\fa\mapsto \lam_i(\sigma_j(\fa))$$
is also an infinite-order Hecke character and hence can be written as $\lam^{\mathbf{m}^{(ij)}}$ for some 
$$\mathbf{m}^{(ij)} = \left(m^{(ij)}_1,\ldots,m^{(ij)}_d\right)\in \ZZ^d.$$ Recalling the definition of $\vec{\phi}(\fa)$ as $\lam_j(\fa) = e^{2\pi i\phi_j(\fa)}$, we see that the condition $\norm{\vec{\phi}(\sigma_j(\fa)) - \vec{\phi}_0} < x^{-\delta}$ is equivalent to $\vec{\phi}(\fa)$ lying inside a polytope in $\TT^d$ cut out by the hyperplanes $$\abs{\phi_{0i}-\arg \prod_{l=1}^d \lam_l^{m^{(ij)}_l}(\fa)} = \abs{\phi_{0i}-\sum_{l=1}^d m^{(ij)}_l\phi_l(\fa)} < x^{-\delta}$$ for $1\le i\le d$. The intersection of such polytopes across all $1\le j\le k$ is a union $\blacktriangle$ of polytopes $\vartriangle$. If $\blacktriangle$ is empty, then the sum in \eqref{equation:von-mangoldt-ideal-class} equals 0 and hence we are done. We restrict to the case in which $\blacktriangle$ is non-empty for the rest of the proof, and we can consider each individual polytope $\vartriangle \subeq \blacktriangle$ independently. For sufficiently large $x$, the individual polytopes $\vartriangle \subeq \blacktriangle$ corresponding to different values of $x$ are all homothetic. In particular, $\vartriangle$ is a dilation by $x^{-\delta}$ of a fixed polytope $\vartriangle_0\subeq \TT^d$ depending only on the choice of automorphisms $\sigma_1<\cdots<\sigma_k$.

Rather than estimate the von Mangoldt sum
\begin{equation} \label{equation:von-mangoldt-sum}
    \sum_{\substack{x\le N\fa < x+h \\ \forall 1\le j\le k \\ \sigma_j(\fa)\in I \\  \norm{\vec{\phi}(\sigma_j(\fa)) - \vec{\phi}_0} < x^{-\delta}}} \Lam(\fa) = \sum_{\fa} \one_{I'}(\fa) \one_{[x, x+h)}(N\fa) \one_{\vartriangle}(\vec{\phi}(\fa)) \Lam(\fa)
\end{equation}
in \eqref{equation:von-mangoldt}, it is more convenient to estimate a smoothed von Mangoldt sum, replacing indicator functions for $x\le N\fa < x+h$ and $\vec{\phi}(\fa)\in \vartriangle$ with smooth counterparts. It is a standard technique to compute a smooth minorant and a smooth majorant and prove that both satisfy the asymptotic on the right hand side of \eqref{equation:x-delta}.

\subsection{Smoothing functions} \label{section:smoothing-functions}

For a short interval $[x, x+h)$ with $h = x^{1-\delta'}$, define $$g_u^-(y) \le \one_{[x, x+h)}(y) \le g_u^+(y),$$ for $u\ll h$ to be chosen later, to be smooth functions supported on $[x-u,x+h+u]$ such that $g_u^-(y) = 1$ for $x+u\le y \le x+h-u$ and $g_u^+(y) = 1$ for $x\le y\le x+h$. In the intervals $[x, x+u]\cup [x+h-u, x+h]$ for $g_u^-(y)$ and $[x-u, x]\cup [x+h, x+h+u]$ for $g_u^+(y)$, we set $g_u^\pm(y)$ to be a translation and horizontal dilation by $u$ of a fixed smooth transition function $g_0(y)$ from 0 to 1 and its reflection from 1 to 0 respectively. More precisely, define $$g_0(y) = \frac{e^{-\frac{1}{y}}}{e^{-\frac{1}{y}} + e^{-\frac{1}{1 - y}}}$$ for $y\in (0, 1)$ and define
\begin{align*}
    g_u^-(y) &=
    \begin{cases} 
        g_0\bigp{\frac{y - x}{u}} &  \textif y\in (x, x+u) \\
        1 & \textif y\in [x + u, x+h-u] \\
        g_0\bigp{\frac{x + h - y}{u}} &  \textif y\in (x+h-u, x+h) \\
        0 & \quad\text{otherwise}
    \end{cases}
\end{align*}
\begin{align*}
    g_u^+(y) &=
    \begin{cases} 
        g_0\bigp{\frac{y - x + u}{u}} &  \textif y\in (x - u, x) \\
        1 & \textif y\in [x, x+h] \\
        g_0\bigp{\frac{x + h + u - y}{u}} &  \textif y\in (x+h, x+h+u) \\
        0 & \quad\text{otherwise}.
    \end{cases}
\end{align*}
We remark that the $\ell$-th derivative of $g_u^\pm(y)$ satisfies $\abs{{g_u^\pm}^{(\ell)}(y)} \ll_\ell u^{-\ell}$ by the chain rule because the functions ${g_u^\pm}^{(\ell)}(y)$ for $\ell>0$ have compact support. The Mellin transform of $g_u^\pm(y)$ is given by the entire function $$G_u^\pm(s) = \int_{0}^\infty g_u^\pm(y) y^{s-1}\,dy$$ of the complex variable $s = \sigma + it$. For any $\ell\in \ZZ_{\ge 0}$ and $s = \sigma + it$ with $\sigma \le 2$, we have the following estimate on $G_u^\pm(s)$:

\begin{equation} \label{equation:mellin-transform-decay}
     \abs{G_u^\pm(s)} \ll_\ell u^{-\ell} (1+\abs{s})^{-\ell} hx^{\sigma+\ell-1}.
\end{equation}

By the Mellin inversion formula, we have $$g_u^\pm(N\fa) = \frac{1}{2\pi i} \int_{2-i\infty}^{2+i\infty} G_u^\pm(s)(N\fa)^{-s}\,ds,$$ where the integral is absolutely convergent because of the rapid decay property of $G_u^\pm(s)$ on $\sigma= 2$.

As discussed previously, we will show that both a smooth minorant and a smooth majorant for \eqref{equation:von-mangoldt-sum} are of the correct asymptotic. For a minorant and majorant of $\one_{\vartriangle}(\vec{\phi}(\fa))$, we use a $d$-dimensional version of the Selberg minorizing and majorizing function for boxes \cite{cochrane}:

\begin{theorem} \label{theorem:trigonometric-approximation}
For any positive integer $M$ and any cube $$B = \prod_{j=1}^d [a_j, b_j]\subeq \TT^d,\quad b_j - a_j = \kappa$$ with $\vol(B)<1$, there exist trigonometric polynomials $$f_M^\pm(\vec{\phi}) = \sum_{\substack{\norm{\mathbf{m}}\le M}} \wh{f_M^\pm}(\mathbf{m}) e^{2\pi i\biga{\mathbf{m}, \vec{\phi}}}$$ such that:
\begin{itemize}
    \item $f_M^-(\vec{\phi})\le \chi_{B}(\vec{\phi})\le f_M^+(\vec{\phi})$ for $\vec{\phi}\in \TT^d$.
    \item $\vol(B) - \wh{f_M^-}(\mathbf{0}) = \bigp{\kappa + \frac{2}{M + 1}}^d - \bigp{\kappa + \frac{1}{M + 1}}^d$
    \item $\wh{f_M^+}(\mathbf{0}) - \vol(B) = \bigp{\kappa + \frac{1}{M + 1}}^d - \kappa^d$
\end{itemize} 
\end{theorem}
\begin{proof}
   This follows immediately from \cite[Theorem 1]{cochrane}.
\end{proof}

We prove a bound on the size of the Fourier coefficients of $f_M^\pm$:

\begin{lemma} \label{lemma:trigonometric-approximation-fourier-coeffient-bound}
For all $\mathbf{m} = (m_1,\ldots,m_d)\in \mathbb{Z}^d$, we have
$$\abs{\widehat{f_M^\pm}(\mathbf{m})} \ll_d \max\bigp{\frac{\kappa^{d-1}}{M} , M^{-d}} + \prod_{j=1}^d \min\bigp{\kappa, \frac{1}{\abs{m_j}}}.$$
\end{lemma}
\begin{proof}
We apply the inequality $\abs{\widehat{f}(\mathbf{m})}\le \norm{f}_{L^1}$ to the function $f(\vec{\phi}) = f_M^\pm(\vec{\phi}) - \chi_B(\vec{\phi})$. We compute
\[
    \abs{\widehat{f_M^\pm}(\mathbf{m}) - \widehat{\chi_B}(\mathbf{m})}
    \le \norm{f_M^\pm - \chi_B}_{L^1}\le \prod_{j=1}^d \bigp{\kappa + \frac{2}{M + 1}} - \kappa^d\ll_d \max\bigp{\frac{\kappa^{d-1}}{M} , M^{-d}}
\]
and
\[
    \abs{\widehat{\chi_B}(\mathbf{m})}= \prod_{j = 1}^d \abs{\frac{e^{2\pi i m_j a_j} - e^{2\pi i m_j b_j}}{2\pi i m_j}}= \prod_{j = 1}^d \abs{\frac{\sin(\pi m_j \kappa)}{\pi m_j}}\le \prod_{j=1}^d \min\bigp{\kappa, \frac{1}{\pi \abs{m_j}}}.
\]
The result follows by the triangle inequality.
\end{proof}

The key idea to smooth $\one_\vartriangle(\vec{\phi}(\fa))$ is to tile the polytope $\vartriangle$ using disjoint cubes. In particular, since $\vartriangle$ is a dilation by $x^{-\delta}$ of $\vartriangle_0$, we will set $\kappa = \kappa_0 x^{-\delta}$ for some $\kappa_0 > 0$ depending on $\vartriangle_0$ to be chosen later, where the asymptotic is with respect to $x\to \infty$. Consider an arbitrary tessellation of $\TT^d$ by a grid of disjoint cubes of side length $\kappa_0$. Such a tessellation can be translated, so without loss of generality fix $\vec{\phi}_0\in \TT^d$ as a vertex of a cube in the grid for all $\kappa_0$. As $\kappa_0\to 0$, the combined volume of cubes that intersect $\vartriangle_0$ will approach $\vol(\vartriangle_0)$.

\begin{lemma}\label{lemma:polytope-tiling}
For a grid of disjoint cubes of side length $\kappa_0$ and an arbitrary solid polytope $\vartriangle_0\subeq \TT^d$, denote by $\boxdot_{\kappa_0}^-$ the union of cubes that lie completely inside $\vartriangle_0$ and denote by $\boxdot_{\kappa_0}^+$ the union of cubes that have non-empty intersection with $\vartriangle_0$. Note that $\boxdot_{\kappa_0}^- \subeq \vartriangle_0 \subeq \boxdot_{\kappa_0}^+$ by definition. Then we have $$\abs{\vol(\boxdot_{\kappa_0}^\pm) - \vol(\vartriangle_0)} = O\bigp{ \kappa_0 \cdot \vol(\p \vartriangle_0)},$$ where $\vol(\p \vartriangle_0)$ denotes the surface area of $\vartriangle_0$.
\end{lemma}

\begin{proof}
This is a well-known fact from geometry using the fact that $\vartriangle_0$ is a polytope.
\end{proof}

In Lemma \ref{lemma:polytope-tiling}, scale the polytope $\vartriangle_0$ and the grid by a factor of $x^{-\delta}$ so that $\vartriangle_0$ becomes $\vartriangle$. Denote the scaled $\boxdot_{\kappa_0}^\pm$ by $\boxdot_{\kappa}^\pm$, which consists of cubes of side length $\kappa$. Then we have
\begin{equation} \label{equation:cube-volume-error-estimate}
    \begin{split}
        \abs{\vol(\boxdot_{\kappa}^\pm) - \vol(\vartriangle)}
        &= \abs{(x^{-\delta})^d \bigp{\vol(\boxdot_{\kappa_0}^\pm) - \vol(\vartriangle_0) }}
        =   O\bigp{ (x^{-\delta})^d\cdot \kappa_0 \cdot \vol(\p \vartriangle_0)}.
    \end{split}
\end{equation}
We majorize the indicator function for each cube $\cC\subeq \boxdot_{\kappa_0}^\pm$ of side length $\kappa$ by a translation ${f_M^\pm}^{\cC}(\vec{\phi})$ of the Selberg majorizing function $f_M^\pm(\vec{\phi})$, and define $$F_M^-(\vec{\phi}) := \sum_{\cC\in \boxdot_{\kappa}^-} {f_M^-}^{\cC} (\vec{\phi})\le \one_{\vartriangle}(\vec{\phi}(\fa))\le \sum_{\cC\in \boxdot_{\kappa}^+} {f_M^+}^{\cC} (\vec{\phi}) =: F_M^+(\vec{\phi}).$$

\subsection{Zero-free region and zero-density estimate for Hecke $L$-functions}
Each Hecke character $\mu \lam^\mathbf{m}\bmod \ff$ has an associated $L$-function defined by
\begin{align*}
    L(s,\mu\lam^\mathbf{m}) &= \sum_{\fa} \mu\lam^\mathbf{m}(\fa) (N\fa)^{-s},
\end{align*}
with logarithmic derivative
\begin{align*}
    -\frac{L'}{L} (s,\mu\lam^\mathbf{m}) &= \sum_{\fa} \Lam(\fa) \mu\lam^\mathbf{m}(\fa) (N\fa)^{-s}.
\end{align*}
If $\mu$ is not primitive then it is induced by a primitive character $\mu_1\bmod \ff_1$ where $\ff_1\mid \ff$, and $\mu_1 \lam^\mathbf{m}$ is a primitive Hecke character mod $\ff_1.$ Furthermore, we have the $L$-function factorization $$L(s,\mu \lam^\mathbf{m}) = L(s, \mu_1 \lam^\mathbf{m}) \prod_{\fp\mid \ff} \bigp{ 1 - \mu_1 \lam^\mathbf{m}(\fp) (N\fp)^{-s}}.$$ The $L$-function for the primitive Hecke character $\mu_1 \lam^\mathbf{m}$ satisfies the functional equation \cite[Theorem 12.3]{iwaniec} given by
\begin{equation}\label{equation:hecke-functional-equation}
L(s,\mu_1 \lam^\mathbf{m}) = w(\mu_1 \lam^\mathbf{m}) A^{\frac{1}{2} - s} L_\infty(1-s, \mathbf{m}, \wh{\mu_1}) L(1-s, \mu_1 \lam^\mathbf{m})
\end{equation}
for some $\abs{w} = 1$, where the following notation is used:
\begin{itemize}
    \item $A = \abs{\Delta_{K}} \cdot (N\ff_1) \cdot \pi^{-n} 2^{-r_2}$.
    \item $\hat{\mu}_1$ is the sign character induced by $\mu_1$, and $$L_\infty(s, \mathbf{m}, \hat{\mu}_1) = \bigp{\prod_{j=1}^{r_1} \frac{\Gamma\bigp{\frac{1}{2}(s+a_j - ib_j)}}{\Gamma\bigp{\frac{1}{2}(1 - s + a_j + ib_j)}}} \bigp{\prod_{j = r_1 + 1}^{r_1 + r_2} \frac{\Gamma\bigp{s + \frac{1}{2} \abs{a_j} - ib_j}}{\Gamma\bigp{1 - s + \frac{1}{2} \abs{a_j} + ib_j}}},$$ where the values $a_1,\ldots,a_{r_1}\in \{0,1\}$ are determined by $\hat{\mu}_1$ and the values $a_{r_1+1},\ldots,a_{r_1+r_2}\in \ZZ$ and $b_1,\ldots,b_{r_1+r_2}\in \RR$ are determined by $\mathbf{m}.$
\end{itemize}
The trivial zeros of $\mu \lam^\mathbf{m}$ occur at the poles of the $\Gamma$ factor and include zeros contributed by the factors $$\prod_{\fp\mid \ff} \bigp{ 1 - \mu_1 \lam^\mathbf{m}(\fp) (N\fp)^{-s}}$$ for non-primitive $\mu$ on $\Re s = 0.$ The first kind are at negative-integer or even-integer translates of a fixed set of complex numbers with non-positive real parts.

In order to get an asymptotic for the smooth minorant and majorant
\begin{equation}\label{equation:smooth-minorant-majorant}
    \begin{split}
        \sum_{\fa} \one_{I'}(\fa) g_u^-(N\fa) F_M^-(\vec{\phi}(\fa))\Lam(\fa)
        &\le \sum_{\fa} \one_{I'}(\fa) \one_{[x, x+h)}(N\fa) \one_{\vartriangle}(\vec{\phi}(\fa)) \Lam(\fa)
        \\&\le \sum_{\fa} \one_{I'}(\fa) g_u^+(N\fa) F_M^+(\vec{\phi}(\fa))\Lam(\fa)
    \end{split}
\end{equation}
in terms of $x$, we will use a suitable zero-free region and a zero-density estimate for the Hecke $L$-functions $L(s,\mu \lam^\mathbf{m})$. The zero-free region is due to Coleman:
\begin{theorem}[{\cite[Theorem 2]{coleman-zero-free-region}}]\label{theorem:coleman-zero-free-region}
Let $\mu \lam^\mathbf{m}\bmod \ff$ be a Hecke character and $L(\sigma + it, \mu \lam^\mathbf{m})$ the associated Hecke $L$-function. Denote $V = \max\bigp{\norm{\mathbf{m}}, t}$. For a constant $A > 0$ depending only on $K$ and sufficiently large $V$, we have 
\begin{equation*}
    L(\sigma + it, \mu \lam^\mathbf{m}) \neq 0 \text{ for } \sigma \ge 1 - \frac{A}{\max\bigp{\log N\ff, (\log V)^{\frac{2}{3}} (\log \log V)^{\frac{1}{3}}}}
\end{equation*}
apart from a possible exceptional zero when $\mathbf{m} = \mathbf{0}$ and $\mu$ is a real character mod $\ff$.
\end{theorem}
Recall that since $\ff = (1)$ in the statement of Theorem \ref{theorem:galois-primes}, the dependence on $N\ff$ in Theorem \ref{theorem:coleman-zero-free-region} can be ignored. We now state a zero-density estimate for Hecke $L$-functions due to Coleman \cite{coleman-zero-free-region}. Define $N(\sigma, T, \mu \lam^\mathbf{m})$ for $\sigma\in [0,1]$ to be the number of zeros of $L(s,\mu \lam^\mathbf{m})$ with $\Re s \ge \sigma$ and $\abs{\Im s} \le T$.

\begin{theorem}[{\cite[Lemma 1]{coleman-general}}] \label{theorem:coleman-zero-density-estimate}
Fix a finite-order Hecke character $\mu\bmod \ff$. There is a constant $B$ depending only on $K$ such that
\begin{equation*} 
       \sum_{\norm{\mathbf{m}} \leq T} N(\sigma, T, \mu \lam^\mathbf{m}) \ll_\mu T^{\frac{5n}{2}(1 - \sigma)}(\log T)^B
    \end{equation*}
for $0 \le \sigma \le 1$. 
\end{theorem}

\subsection{Explicit formula}

We return now to estimating the smooth minorant and majorant \eqref{equation:smooth-minorant-majorant}. We rewrite each of the terms that appears in the left hand side of \eqref{equation:smooth-minorant-majorant} as follows. Using the orthogonality relations for narrow ideal class characters, we have $$\one_{I'}(\fa) = \frac{1}{\abs{\cI}} \sum_{\eta\bmod (1)} \bar{\eta}(I')\eta(\fa),$$ where $\cI$ denotes the narrow ideal class group and $\eta\bmod (1)$ runs over all $\abs{\cI}$ narrow ideal class characters. Using the Mellin inversion formula for $g_u^\pm(y)$, we have $$g_u^\pm(N\fa) = \frac{1}{2\pi i} \int_{2-i\infty}^{2+i\infty} G_u^\pm(s)(N\fa)^{-s}\,ds.$$ Finally, expand $F_M^\pm$ in its Fourier series $$F_M^\pm(\vec{\phi}(\fa)) = \sum_{\substack{\mathbf{m}\in \ZZ^d \\ \norm{\mathbf{m}}\le M}} \wh{F_M^\pm}(\mathbf{m})\lam^{\mathbf{m}}(\fa),\quad \lam^\mathbf{m}(\fa) = \prod_{j=1}^d \lam_j^{m_j}(\fa).$$ Hence we have
\begin{align}
    &\sum_{\fa} \mathbf{1}_{I'}(\fa) g_u^\pm(N\fa) F_M^\pm(\vec{\phi}(\fa))\Lam(\fa) \nonumber
    \\&= \sum_{\fa} \frac{1}{\abs{\cI}} \sum_{\eta\bmod (1)} \bar{\eta}(I')  \sum_{\substack{\mathbf{m}\in \ZZ^d \\ \norm{\mathbf{m}}\le M}} \wh{F_M^\pm}(\mathbf{m}) \frac{1}{2\pi i} \int_{2-i\infty}^{2+i\infty} G_u^\pm(s) \Lam(\fa) \eta \lam^\mathbf{m}(\fa) (N\fa)^{-s}\,ds \nonumber
    \\&= \frac{1}{\abs{\cI}} \sum_{\eta\bmod (1)} \bar{\eta}(I')  \sum_{\substack{\mathbf{m}\in \ZZ^d \\ \norm{\mathbf{m}}\le M}} \wh{F_M^\pm}(\mathbf{m}) \frac{1}{2\pi i} \int_{2-i\infty}^{2+i\infty} G_u^\pm(s)\bigp{-\frac{L'}{L}(s,\eta\lam^\mathbf{m})}\,ds. \label{equation:von-mangoldt-sum-expansion}
\end{align}

We now express the integral in \eqref{equation:von-mangoldt-sum-expansion} as a sum over the zeros and poles of the Hecke characters $\eta \lam^\mathbf{m}$ for $\norm{\mathbf{m}}\le M$. We use the following classical result on the vertical distribution of zeros of $L(s,\mu \lam^\mathbf{m})$:

\begin{lemma} \label{lemma:hecke-vertical-zero-distribution-gap}
For any $\abs{T}$ sufficiently large, there are $\ll \log (\abs{T}+{M}) + \log(N\ff)$ non-trivial zeros $\rho$ of $L(s,\mu \lam^\mathbf{m})$ for $\norm{\mathbf{m}}\le M$ such that $\abs{T - \Im \rho}<1$. In particular, there is a gap in the vertical distribution of zeros of length $\gg \bigp{\log (\abs{T}+{M}) + \log(N\ff)}^{-1}$.
\end{lemma}

\begin{proof}
    This follows from \cite[Proposition 5.7]{iwaniec-kowalski}.
\end{proof}

A standard argument using contour integration analogous to the proof of \cite[Lemma 3.3]{rouse} yields the smoothed explicit formula
\begin{equation} \label{equation:von-mangoldt-explicit-formula}
\begin{split}
    &\sum_{\fa} \one_{I'}(\fa) g_u^\pm(N\fa) F_M^\pm(\vec{\phi}(\fa))\Lam(\fa)
    \\&= \frac{1}{\abs{\cI}} \bigp{\widehat{F_M^\pm}(0)G_u^\pm(1) - \sum_{\eta\bmod (1)} \bar{\eta}(I') \sum_{\substack{\mathbf{m}\in \ZZ^d \\ \norm{\mathbf{m}}\le M}}  \widehat{F_M^\pm}(\mathbf{m}) \sum_{L(\rho, \eta\lam^\mathbf{m}) = 0} G_u^\pm(\rho)}.
\end{split}
\end{equation}
Here we have used the fact that $L(s, \eta \lam^\mathbf{m})$ has a pole at $s=1$ if and only if $\eta = 1$ and $\mathbf{m} = \mathbf{0}.$ We split \eqref{equation:von-mangoldt-explicit-formula} into the main term $\widehat{F_M^\pm}(0)G_u^\pm(1)$ and the sum over zeros $$\sum_{\eta\bmod (1)} \bar{\eta}(I') \sum_{\substack{\mathbf{m}\in \ZZ^d \\ \norm{\mathbf{m}}\le M}}  \widehat{F_M^\pm}(\mathbf{m}) \sum_{L(\rho, \eta\lam^\mathbf{m}) = 0} G_u^\pm(\rho).$$

\subsection{Estimating the main term} \label{section:estimating-main-term}
First we estimate the main term $\widehat{F_M^\pm}(0)G_u^\pm(1)$. Recall the definitions of $\vartriangle_0$, $\kappa$, and $\kappa_0$ from Section \ref{section:principle-inclusion-exclusion} and Section \ref{section:smoothing-functions}. We use Theorem \ref{theorem:trigonometric-approximation}, noting that there are $\vol(\boxdot_{\kappa}^\pm)\kappa^{-d}$ cubes $\cC$ in $\boxdot_\kappa^\pm$. By \eqref{equation:cube-volume-error-estimate}, we compute
\begin{align*}
    \widehat{F_M^\pm}(0)
    &= \vol(\boxdot_{\kappa}^\pm)\left(1 +  O\bigp{ \kappa^{-d}\max\bigp{\frac{\kappa^{d-1}}{M} ,M^{-d}}}\right)
    \\&= (x^{-\delta})^d\bigp{\vol(\vartriangle_0) + O\bigp{\kappa_0 \cdot \vol(\p \vartriangle_0)}} \bigp{1 +  O\bigp{\kappa^{-d}\max\bigp{\frac{\kappa^{d-1}}{M}, M^{-d}}}}
    \\&= (x^{-\delta})^d \vol(\vartriangle_0)\bigp{1 + O_{\vartriangle_0} \bigp{\kappa_0 + \kappa^{-d} \max\bigp{\frac{\kappa^{d-1}}{M}, M^{-d}}}}
\end{align*}
using the fact that $\vol(\vartriangle_0)$ and $\vol(\p \vartriangle_0)$ are constants that only depend on the choice of automorphisms $\sigma_1<\cdots<\sigma_k$. From the definition of the Mellin transform $G_u^\pm$ we have $$G_u^\pm(1) = \int_{0}^\infty g_u^\pm(y)\, dy = h + O(u).$$ Combining yields
\begin{equation}\label{equation:explicit-formula-error-main-term}
    \widehat{F_M^\pm}(0)G_u^\pm(1) = h(x^{-\delta})^d \vol(\vartriangle_0)\left(1 + O_{\vartriangle_0}\bigp{ \frac{u}{h} + \kappa_0 + \frac{1}{\kappa M}}\right)
 \end{equation}
so long as $u = o(h), \kappa_0 = o(1)$, and $(\kappa M)^{-1} = o(1)$, where the asymptotics are with respect to $x\to \infty$.

\subsection{Estimating the sum over zeros} \label{section:estimating-sum-over-zeros}

Now we estimate the sum over zero. Using Lemma \ref{lemma:trigonometric-approximation-fourier-coeffient-bound} with $M\gg \kappa^{-1}$ and $\kappa_0 \ll 1$, along with the fact that there are $\vol(\boxdot_{\kappa}^\pm)\kappa^{-d}$ cubes $\cC$ in $\boxdot_\kappa$, we compute
\begin{equation} \label{equation:explicit-formula-sum-over-zeros}
    \begin{split}
        &\frac{1}{\abs{\cI}} \sum_{\eta\bmod (1)} \bar{\eta}(I') \sum_{\substack{\mathbf{m}\in \ZZ^d \\ \norm{\mathbf{m}}\le M}}  \widehat{F_M^\pm}(\mathbf{m}) \sum_{L(\rho, \eta\lam^\mathbf{m}) = 0} G_u^\pm(\rho)
        \\&\ll \frac{1}{\abs{\cI}} \sum_{\eta\bmod (1)} \sum_{\substack{\mathbf{m}\in \ZZ^d \\ \norm{\mathbf{m}}\le M}}  \abs{\widehat{F_M^\pm}(\mathbf{m})} \sum_{L(\rho, \eta\lam^\mathbf{m}) = 0} \abs{G_u^\pm(\rho)}
        \\&\ll \max_{\eta\bmod (1)}  \sum_{\substack{\mathbf{m}\in \ZZ^d \\ \norm{\mathbf{m}}\le M}}  \bigp{\vol(\boxdot_{\kappa}^\pm)\kappa^{-d}}\kappa^{d} \sum_{L(\rho, \eta\lam^\mathbf{m}) = 0} \abs{G_u^\pm(\rho)}
        \\&\ll \max_{\eta\bmod (1)} \sum_{\substack{\mathbf{m}\in \ZZ^d \\ \norm{\mathbf{m}}\le M}} (x^{-\delta})^d \vol(\vartriangle_0) \bigp{1+O_{\vartriangle_0}(\kappa_0)} \sum_{L(\rho, \eta\lam^\mathbf{m}) = 0} \abs{G_u^\pm(\rho)}
        \\&\ll_{\vartriangle_0} (x^{-\delta})^d \max_{\eta\bmod (1)}  \sum_{\substack{\mathbf{m}\in \ZZ^d \\ \norm{\mathbf{m}}\le M}}  \sum_{L(\rho, \eta\lam^\mathbf{m}) = 0} \abs{G_u^\pm(\rho)}.
    \end{split}
\end{equation}
Fix a narrow class character $\eta\bmod (1)$. We will split the sum over zeros $\rho$ in those with $\abs{\Im \rho} \le M$ and those with $\abs{\Im \rho} > M$.

We first bound the sum over zeros with $\abs{\Im \rho} > M$. By \eqref{equation:mellin-transform-decay} we have $$G_u^\pm(\rho) \ll_\ell u^{-\ell} \abs{\Im \rho}^{-\ell} hx^{\ell}.$$
Since the trivial zeros on $\Im \rho = 0$ of $\eta$ are evenly spaced, the rapid decay of $G_u^\pm(\rho)$ guarantees that both sets of trivial zeros contribute negligibly to the sum in \eqref{equation:explicit-formula-sum-over-zeros} compared to our eventual bound for the sum over the non-trivial zeros. For the rest of the argument, we restrict to a sum over non-trivial zeros $\rho$ of $L(s,\eta \lam^\mathbf{m})$. Using Lemma \ref{lemma:hecke-vertical-zero-distribution-gap} for the vertical distribution of non-trivial zeros, we have
\begin{equation} \label{equation:explicit-formula-error-term-high-imaginary-part}
    \begin{split}
        \sum_{\substack{\mathbf{m}\in \ZZ^d \\ \norm{\mathbf{m}}\le M}} \sum_{\substack{L(\rho, \eta\lam^\mathbf{m}) = 0 \\ \abs{\Im \rho} > M}} \abs{G_u^\pm(\rho)}
        &\ll_{\ell} \int_{M}^\infty  u^{-\ell} t^{-\ell} hx^{\ell} M^d\log(\abs{t} + M)\,dt
        \\&\ll_\ell (x^{-\delta})^d u^{-\ell} M^{-\ell + 1 + d} hx^\ell \log M
    \end{split}
\end{equation}

Next, we bound the sum over zeros with $\abs{\Im \rho} \le M$. Taking $\ell = 0$ in \eqref{equation:mellin-transform-decay}, we have $G_u^\pm(\rho) \ll hx^{\Re \rho - 1}$. Hence we obtain
\begin{align*}
    \sum_{\substack{\mathbf{m}\in \ZZ^d \\ \norm{\mathbf{m}}\le M}} \sum_{\substack{L(\rho, \eta\lam^\mathbf{m}) = 0 \\ \abs{\Im \rho} \le M}} \abs{G_u^\pm(\rho)}
    &\ll \sum_{\substack{\mathbf{m}\in \ZZ^d \\ \norm{\mathbf{m}}\le M}}  \sum_{\substack{L(\rho, \eta\lam^\mathbf{m}) = 0 \\ \abs{\Im \rho} \le M}} hx^{\Re \rho - 1}.
\end{align*}

\begin{lemma}\label{lemma:explicit-formula-error-term-small-imaginary-part}
If $M = x^{\tau}$ with $0 < \tau < \frac{2}{5n}$, then
$$\sum_{\substack{\mathbf{m}\in \ZZ^d \\ \norm{\mathbf{m}}\le M}} \sum_{\substack{L(\rho, \eta\lam^\mathbf{m}) = 0 \\ \abs{\Im \rho} \le M}} x^{\Re \rho - 1} \ll e^{-c_1 (\log x)^{c_2}}.$$
for some constants $c_1, c_2 > 0$ that depend on $\tau$.
\end{lemma}
\begin{proof}
We compute using partial summation that

\begin{equation} \label{equation:partial-summation-short-intervals}
    \begin{split}
        &\sum_{\substack{\mathbf{m}\in \ZZ^d \\ \norm{\mathbf{m}}\le M}} \sum_{\substack{L(\rho, \eta\lam^\mathbf{m}) = 0 \\ \abs{\Im \rho} \le M}} x^{\Re \rho - 1}
        \\&= \sum_{\substack{\mathbf{m}\in \ZZ^d \\ \norm{\mathbf{m}}\le M}} \int_{0}^1 x^{\sigma - 1}\,d N(\sigma, M, \eta \lam^\mathbf{m})
        \\&= \sum_{\substack{\mathbf{m}\in \ZZ^d \\ \norm{\mathbf{m}}\le M}} \bigp{N(0, M, \eta \lam^\mathbf{m})x^{-1}\log x + \int_{0}^1  N(\sigma, M, \eta \lam^\mathbf{m})x^{\sigma - 1} \log x\,d\sigma}
        \\&\ll \sup_{\sigma\in [0,1]}\sum_{\substack{\mathbf{m}\in \ZZ^d \\ \norm{\mathbf{m}}\le M}} N(\sigma, M, \eta \lam^\mathbf{m})x^{\sigma - 1} \log x.
    \end{split}
\end{equation}
By Theorem \ref{theorem:coleman-zero-free-region} and Theorem \ref{theorem:coleman-zero-density-estimate}, the right hand side of \eqref{equation:partial-summation-short-intervals} is
\begin{align*}
    &\ll \max\bigp{(\log x)^{B+1} \cdot  \sup_{\sigma\in \bigb{\frac12, 1-A(\log M)^{-\frac{2}{3}} (\log \log M)^{-\frac{1}{3}}}} \bigp{x^{\sigma - 1} M^{\frac{5n}{2}(1-\sigma)}}, x^{\beta - 1}\log x}
    \\&\ll (\log x)^{B+1} \cdot  x^{-\frac{A\bigp{1 - \frac{5n}{2}\tau}}{(\log M)^{\frac{2}{3} + o(1)}}}
    \\&\ll e^{-A\tau^{-\frac{2}{3}}\bigp{1 - \frac{5n}{2}\tau} (\log x)^{\frac{1}{3} - o(1)}}
\end{align*}
where $\beta < 1$ denotes the possible exceptional zero for $\mathbf{m} = 0$ if $\eta$ is real. We thus deduce Lemma \ref{lemma:explicit-formula-error-term-small-imaginary-part} with $c_1 = A\tau^{-\frac{2}{3}}\bigp{1 - \frac{5n}{2}\tau}$ and $0<c_2 < \frac{1}{3}$.
\end{proof}

\subsection{Combining the error estimates}\label{section:combining-error-estimates}

It is now time to collect all of the error terms. The error from the main term is computed in \eqref{equation:explicit-formula-error-main-term}. The error from the zeros with $\abs{\Im \rho} > M$ is computed in \eqref{equation:explicit-formula-error-term-high-imaginary-part}. The error from the zeros with $\abs{\Im \rho}\le M$ is computed in Lemma \ref{lemma:explicit-formula-error-term-small-imaginary-part}. The sum of the error terms is bounded by

\begin{align*}
    &\ll_{\ell, \vartriangle_0} \ub{h(x^{-\delta})^d \cdot \bigp{ \frac{u}{h} + \kappa_0 + \frac{1}{\kappa M}}}_\text{main term error} + \ub{(x^{-\delta})^d \cdot  u^{-\ell} M^{-\ell + 1 + d} hx^\ell \log M}_\text{sum over $\abs{\Im \rho} > M$} + \ub{(x^{-\delta})^d  \cdot h e^{-c_1 (\log x)^{c_2}}}_\text{sum over $\abs{\Im \rho} \le M$}
    \\&\ll_{\ell,\vartriangle_0} h(x^{-\delta})^d \bigp{\frac{u}{h} + \kappa_0 + \frac{x^{\delta}}{\kappa_0 M} + u^{-\ell} x^{\ell} M^{-\ell + 1 + d} \log M +  e^{-c_1(\log x)^{c_2}}}.
\end{align*}

Recall that the parameter $\delta$ corresponds to the sector width and $\delta'$ corresponds to the short interval $h = x^{1-\delta'}$ in the statement of Theorem \ref{theorem:galois-primes}. Given $\delta, \delta'<\frac{2}{5n}$, we now choose values for the parameters $u$, $M$, and $\kappa_0$:
\begin{equation*}
    u = x^{1-\tau'},\quad M = x^{\tau}, \quad \kappa_0 = e^{-c_1 (\log x)^{c_2}}, \quad \delta,\delta' < \tau' < \tau < \frac{2}{5n}.
\end{equation*}
Note that the choice of $\tau$ determines the value $c_1$ in the proof of Lemma \ref{lemma:explicit-formula-error-term-small-imaginary-part}, and $c_2$ can be chosen to be any constant in the range $0<c_2<\frac{1}{3}$. The sum of the error terms is bounded by
\begin{align*}\label{equation:explicit-formula-error-sum}
    &\ll_{\ell,\vartriangle_0} h(x^{-\delta})^d \bigp{\frac{u}{h} + \kappa_0 + \frac{x^{\delta}}{\kappa_0 M} + u^{-\ell} x^{\ell} M^{-\ell + 1 + d} \log M +  e^{-c_1(\log x)^{c_2}}}
    \\&\ll_{\ell,\vartriangle_0} h(x^{-\delta})^d \bigp{ x^{-(\tau' - \delta')} + e^{-c_1 (\log x)^{c_2}} + e^{c_1 (\log x)^{c_2}}\cdot x^{-(\tau - \delta)} + x^{-\ell (1-\tau') + \ell + \tau(-\ell + 1 + d) + o(1)}}
    \\&\ll_{\ell,\vartriangle_0} h(x^{-\delta})^d \bigp{ e^{-c_1 (\log x)^{c_2}} + x^{-\ell( \tau - \tau')+ \tau(d+1) + o(1)}}.
\end{align*}
Finally, choose $\ell$ large enough so that $$\ell > \frac{\tau(d+1)}{\tau - \tau'},$$ in which case $-\ell( \tau - \tau')+ \tau(d+1) + o(1)$ is negative and uniformly bounded away from 0 for all sufficiently large $x$. We conclude that the sum of the error terms is bounded by $$\ll  h(x^{-\delta})^d \cdot e^{-c_1 (\log x)^{c_2}}.$$

In conclusion, we have
\begin{align*}
    \sum_{\fa} \one_{I'}(\fa) g_u^\pm(N\fa) f_M^\pm(\vec{\phi}(\fa))\Lam(\fa)
    &= \frac{h \vol(\vartriangle)}{\abs{\cI}}\bigp{1 + O\bigp{e^{-c_1(\log x)^{c_2}}}}
    \\&= \frac{h \vol(\vartriangle_0)x^{-(n-1)\delta}}{\abs{\cI}} \bigp{1 + O\bigp{e^{-c_1(\log x)^{c_2}}}},
\end{align*}
from which we deduce $$\sum_{\fa} \one_{I'}(\fa) \one_{[x, x+h)}(N\fa) \one_{\vartriangle}(\vec{\phi}(\fa)) \Lam(\fa) = \frac{\vol(\vartriangle_0) h x^{-(n-1)\delta}}{\abs{\cI}} \bigp{1 + O\bigp{e^{-c_1(\log x)^{c_2}}}}.$$ By summing at most $2^n$ such asymptotics with different polytopes $\vartriangle_0$, as discussed in Section \ref{section:principle-inclusion-exclusion}, \eqref{equation:x-delta} follows.

Finally, to deduce Theorem \ref{theorem:galois-primes} from \eqref{equation:x-delta}, it suffices to slightly change the smoothing functions $F_M^\pm$. The majorant $F_M^+$ should be defined with respect to the original sector condition $\norm{\vec{\phi}(\mathfrak{p})-\vec{\phi}_0}<x^{-\delta}$ as usual. However, the minorant $F_M^-$ should be defined with respect to the sector condition $\norm{\vec{\phi}(\mathfrak{p})-\vec{\phi}_0}<(x + h)^{-\delta}$. The only difference in the proof is a change in the non-error part of the main term $\widehat{F_M^\pm}(0)$ from $(x^{-\delta})^d \vol(\vartriangle_0)$ to $\bigp{(x+h)^{-\delta}}^d\vol(\vartriangle_0)$. The error incurred is $\ll_{\vartriangle_0} (x^{-\delta})^d e^{-c_1(\log x)^{c_2}}$ since $h = x^{1 - \delta'}$ for $\delta'>0$, and hence the error term in \eqref{equation:x-delta} is preserved.

\section{Proof of Theorem \ref{theorem:bombieri-vinogradov-galois-primes}}

As in the proof of Theorem \ref{theorem:galois-primes}, we will instead show the inequality
\begin{equation} \label{equation:bombieri-vinogradov-x-delta}
\sum_{\substack{q \leq Q \\ H_K^+\cap \QQ(\zeta_q) = \QQ }}\max_{\gcd(a, q) = 1}\Bigg|\sum_{\substack{ x\le p < x+h \\ p \equiv a\bmod q\\ \exists \fp \in I \text{ s.t.} \\ N\fp = p \\ \|\vec{\phi}(\fp) - \vec{\phi}_0\|< x^{-\delta} }} \log p - \frac{chx^{-(n-1)\delta}}{\vhi(q) \abs{\cI}}\Bigg| \ll_{A} \frac{hx^{-(n-1)\delta}}{(\log x)^A},
\end{equation}
noting that Theorem \ref{theorem:bombieri-vinogradov-galois-primes} can be deduced from \eqref{equation:bombieri-vinogradov-x-delta} the same way as in the last paragraph of Section \ref{section:proof-galois-primes}. The proof will be analogous to that of the original Bombieri-Vinogradov inequality using a zero-density estimate for the relevant $L$-functions. By the principle of inclusion-exclusion in Section \ref{section:principle-inclusion-exclusion}, it suffices to show that for all automorphisms $\sigma_1<\cdots < \sigma_k \in \Gal(K/\QQ)$ we have
$$\sideset{}{'}\sum_{\substack{q \leq Q}}\max_{\gcd(a, q) = 1}\Bigg| \sum_{\substack{x\le N\fp < x+h \\ N\fp \equiv a \bmod q \\ \forall 1\le j\le k \\ \sigma_j(\fp)\in I \\  N(\sigma_j(\fp)) = p \\ \norm{\vec{\phi}(\sigma_j(\fp)) - \vec{\phi}_0} < x^{-\delta}}} \log (N\fp) - \frac{\vol(\vartriangle_0) hx^{-(n-1)\delta}}{\vhi(q) \abs{\cI}}\Bigg| \ll_A \frac{hx^{-(n-1)\delta}}{(\log x)^A},$$ as Theorem \ref{theorem:bombieri-vinogradov-galois-primes} will follow by summing at most $2^n$ such asymptotics with different polytopes $\vartriangle_0$ and using the triangle inequality.

Because there are at most $\ll x^{\frac12} q^{-\frac12}$ ideals $\fa$ satisfying $N\fa = p^k$ with $k\ge 2$ and $N\fa \equiv a\bmod q$ in the interval $x\le N\fa < x+h$, we have
\begin{equation}\label{equation:von-mangoldt-arithmetic-progressions}
    \sum_{\substack{x\le N\fp < x+h \\ N\fp \equiv a \bmod q \\ \forall 1\le j\le k \\ \sigma_j(\fp)\in I \\  N(\sigma_j(\fp)) = p \\ \norm{\vec{\phi}(\sigma_j(\fp)) - \vec{\phi}_0} < x^{-\delta}}} \log (N\fp) = O\bigp{x^{\frac12} q^{-\frac12}  \log x} + \sum_{\substack{x\le N\fa < x+h \\ N\fa \equiv a \bmod q \\ \forall 1\le j\le k \\ \sigma_j(\fa)\in I  \\ \norm{\vec{\phi}(\sigma_j(\fa)) - \vec{\phi}_0} < x^{-\delta}}} \Lam(\fa).
\end{equation}
Summing over all $q\le Q = x^{\theta}$, the accumulated error is $\ll x^{\frac12} Q^{\frac12} \log x$, which is negligible to the right hand side of Theorem \ref{theorem:bombieri-vinogradov-galois-primes} for our range of $\delta$, $\delta'$, and $\theta$. It thus suffices to show
$$\sideset{}{'}\sum_{\substack{q \leq Q}}\max_{\gcd(a, q) = 1}\Bigg| \sum_{\substack{x\le N\fa < x+h \\ N\fa \equiv a \bmod q \\ \forall 1\le j\le k \\ \sigma_j(\fa)\in I  \\ \norm{\vec{\phi}(\sigma_j(\fa)) - \vec{\phi}_0} < x^{-\delta}}} \Lam(\fa) - \frac{\vol(\vartriangle_0) hx^{-(n-1)\delta}}{\vhi(q) \abs{\cI}}\Bigg| \ll_A \frac{hx^{-(n-1)\delta}}{(\log x)^A}.$$

\subsection{Explicit formula} \label{section:refined-explicit-formula}
We first derive a refinement to \eqref{equation:von-mangoldt-explicit-formula} that imposes an additional constraint that the prime $p$ lies in a fixed residue class $a\bmod q$ with $\gcd(a,q) = 1$ for some modulus $q$ relatively prime to $\Delta_{K}$. We remark that the condition $\gcd(q, \Delta_{K}) = 1$ is necessary to achieve equidistribution among the $\vhi(q)$ relatively prime residue classes mod $q.$ As an example, let $K = \QQ(i), \Delta_{K} = -4$ and let $I$ denote the trivial ideal class. Every prime $p$ represented by the norm of a principal ideal is of the form $a^2 + b^2$ for $a,b\in \ZZ$. When $p$ is odd, it must be the case that $p\equiv 1 \bmod 4$, so equidistribution mod $q$ for any $q \equiv 0\bmod 4$ is not achieved. The condition $\gcd(q, \Delta_{K}) = 1$ rules out these $q$.

We have the following refined explicit formula, with the same definitions of the smoothing functions $g_u^\pm$ and $F_M^\pm$ as in Section \ref{section:smoothing-functions} and the additional notation $\chi\eta \lam^\mathbf{m} = \chi(N\fa)\eta(\fa)\lam^\mathbf{m}(\fa)$:
\begin{equation} \label{equation:von-mangoldt-explicit-formula-arithmetic-progressions}
    \begin{split}
        &\sum_{\fa} \one_{I'}(\fa) \one_{a\bmod q}(N\fa) g_u^\pm (N\fa) F_M^\pm (\vec{\phi}(\fa))\Lam(\fa)
        \\&= \frac{1}{\vhi(q) \abs{\cI}} \bigp{\widehat{F_M^\pm}(0)G_u^\pm(1) - \sum_{\chi\bmod q} \bar{\chi}(a) \sum_{\eta\bmod (1)} \bar{\eta}(I') \sum_{\substack{\mathbf{m}\in \ZZ^d \\ \norm{\mathbf{m}}\le M}}  \widehat{F_M^\pm}(\mathbf{m}) \sum_{L(\rho ,\chi \eta \lam^\mathbf{m}) = 0} G_u^\pm(\rho)}.
    \end{split}
\end{equation}
Using the orthogonality relations for Dirichlet characters, we have written $$\one_{a\bmod q}(N\fa) = \frac{1}{\vhi(q)} \sum_{\chi\bmod q} \bar{\chi}(a)\chi(N\fa),$$ where the sum runs over all $\vhi(q)$ Dirichlet characters modulo $q$. Hence the only change from \eqref{equation:von-mangoldt-explicit-formula} to \eqref{equation:von-mangoldt-explicit-formula-arithmetic-progressions} is the addition of the twist of the Hecke character $\eta \lam^\mathbf{m}$ by the finite-order Hecke character $\chi(N\fa)$. The last fact needed to justify \eqref{equation:von-mangoldt-explicit-formula-arithmetic-progressions} is that $L(s,\chi \eta \lam^\mathbf{m})$ has a pole at $s = 1$ if and only if $\chi=1$ is the trivial Dirichlet character mod $q$, $\eta = 1$ is the trivial narrow ideal class character mod $(1)$, and $\mathbf{m} = \mathbf{0}$. This is guaranteed by the condition $H_K^+\cap \QQ(\zeta_q) = \QQ$, as shown in \cite{murty-petersen}.

\subsection{Reduction to a sum over zeros}

The next step is to use the explicit formula from \eqref{equation:von-mangoldt-explicit-formula-arithmetic-progressions}. As in Section \ref{section:combining-error-estimates}, we choose
\begin{equation*}
    u = x^{1-\tau'},\quad M = x^{\tau}, \quad \delta,\delta' < \tau' < \tau,\quad 2\theta + \tau < \frac{2}{5n}.
\end{equation*}
By \eqref{equation:mellin-transform-decay} and the error estimates in Section \ref{section:estimating-main-term} and Section \ref{section:estimating-sum-over-zeros}, it follows that for $\kappa_0\ll_A (\log x)^{-A - 1}$, we have

\begin{equation} \label{equation:sum-over-zeros-1}
    \begin{split}
        &\sideset{}{'}\sum_{\substack{q \leq Q}} \max_{\gcd(a, q) = 1}\Bigg| \sum_{\substack{x\le N\fa < x+h \\ N\fa \equiv a \bmod q \\ \forall 1\le j\le k \\ \sigma_j(\fa)\in I  \\ \norm{\vec{\phi}(\sigma_j(\fa)) - \vec{\phi}_0} < x^{-\delta}}} \Lam(\fa) - \frac{\vol(\vartriangle_0) hx^{-(n-1)\delta}}{\vhi(q) \abs{\cI}}\Bigg|
        \\&\ll_A \sideset{}{'}\sum_{\substack{q \leq Q}} \frac{hx^{-(n-1)\delta}}{\vhi(q) \abs{\cI}} \bigp{x^{-(\tau' - \delta')} + (\log x)^{-A-1} + (\log x)^{A+1} x \cdot x^{-(\tau - \delta)}}
        \\&+ \sideset{}{'}\sum_{\substack{q \leq Q}} \frac{x^{-(n-1)\delta}}{\vhi(q) \abs{\cI}} \sum_{\chi\bmod q} \sum_{\eta\bmod (1)} \sum_{\substack{\mathbf{m}\in \ZZ^d \\ \norm{\mathbf{m}}\le M}}  \sum_{L(\rho ,\chi \eta \lam^\mathbf{m}) = 0} \abs{G_u^\pm(\rho)}
        \\&\ll_{A}
        \sideset{}{'}\sum_{\substack{q \leq Q}} \frac{hx^{-(n-1)\delta}}{\vhi(q)\abs{\cI}} \bigp{ (\log x)^{-A-1} +  \sum_{\chi\bmod q} \sum_{\eta\bmod (1)} \sum_{\substack{\mathbf{m}\in \ZZ^d \\ \norm{\mathbf{m}}\le M}}  \sum_{L(\rho ,\chi \eta \lam^\mathbf{m}) = 0} \frac{\abs{G_u^{\pm}(\rho)}}{h}}.
    \end{split}
\end{equation}
Using the well-known inequalities 
$$\frac{1}{\phi(q)} \ll \frac{\log\log q}{q}$$
and
$$\quad \sum_{q \leq Q}\frac{1}{\phi(q)} \ll \log Q,$$ we deduce that \eqref{equation:sum-over-zeros-1} is
\begin{equation*} \label{equation:sum-over-zeros-2}
    \begin{split}
        \ll_A \frac{hx^{-(n-1)\delta}}{\abs{\cI}} \bigp{ (\log x)^{-A}+ \sideset{}{'}\sum_{\substack{q \leq Q}} \frac{\log \log q}{q} \sum_{\chi\bmod q} \sum_{\eta\bmod (1)} \sum_{\substack{\mathbf{m}\in \ZZ^d \\ \norm{\mathbf{m}}\le M}}  \sum_{L(\rho ,\chi \eta \lam^\mathbf{m}) = 0} \frac{\abs{G_u^{\pm}(\rho)}}{h}}.
    \end{split}
\end{equation*}
Hence to prove Theorem \ref{theorem:bombieri-vinogradov-galois-primes}, it suffices to show that
\begin{equation} \label{equation:sum-over-zeros-inequality}
    \sideset{}{'}\sum_{\substack{q \leq Q}} \frac{1}{q} \sum_{\chi\bmod q} \sum_{\eta\bmod (1)} \sum_{\substack{\mathbf{m}\in \ZZ^d \\ \norm{\mathbf{m}}\le M}}  \sum_{L(\rho ,\chi \eta \lam^\mathbf{m}) = 0} \frac{\abs{G_u^{\pm}(\rho)}}{h} \ll_A (\log x)^{-A-1},
\end{equation}
where we absorb the factor of $\log \log q \le \log \log Q$ into the right hand side and treat $\abs{\cI}$ as a constant since the narrow ideal class group is fixed by the number field $K$.

\subsection{Reduction to non-trivial zeros of primitive characters}

The goal of this section will be to reduce the sum in \eqref{equation:sum-over-zeros-inequality} to a sum over primitive Hecke characters. Since $\chi \eta = 1$ is principal if and only if $\chi = 1$ and $\eta = 1$ simultaneously, for a fixed $q$ the finite-order Hecke characters $\chi \eta$ are induced by pairwise distinct primitive Hecke characters. The conductor of $\chi$ divides $(q)$ and the conductor of $\eta$ is $(1)$, so the conductor $\fa$ of $\mu = \chi \eta$ divides $(q)$ and hence $N\fa\le q^n$. Since the trivial zeros on $\Im s = 0$ of $\eta$ are evenly spaced for $\eta$ and the trivial zeros on $\Re s = 0$ of non-primitive $\eta$ are evenly spaced as well, the rapid decay of $G_u^\pm(s)$ guarantees that both sets of trivial zeros contribute negligibly to the sum in \eqref{equation:sum-over-zeros-inequality} compared to our eventual bound of the sum over the non-trivial zeros. We thus have
\begin{equation} \label{equation:reduction-to-non-trivial-zeros}
    \begin{split}
        \sideset{}{'}\sum_{\substack{q \leq Q}} \frac{1}{q} \sum_{\chi\bmod q} \sum_{\eta\bmod (1)} \sum_{\substack{\mathbf{m}\in \ZZ^d \\ \norm{\mathbf{m}}\le M}}  \sum_{L(\rho ,\chi \eta \lam^\mathbf{m}) = 0} \frac{\abs{G_u^{\pm}(\rho)}}{h} &\le \sideset{}{'}\sum_{\substack{q \leq Q}} \frac{1}{q} \sum_{\mu \bmod (q)} \sum_{\substack{\mathbf{m}\in \ZZ^d \\ \norm{\mathbf{m}}\le M}}  \sum_{L(\rho ,\mu \lam^\mathbf{m}) = 0} \frac{\abs{G_u^{\pm}(\rho)}}{h}
    \end{split}
\end{equation}
where the inner sum is over non-trivial zeros $\rho$ of $L(s,\mu \lam^\mathbf{m})$.

A primitive Hecke character $\mu_1 \bmod \fa_1$ induces some number of Hecke characters with modulus $(q),\,q\le Q$. Each such $q$ must satisfy $N(q) = q^n\mid N\fa_1$, hence we can bound $$\sum_{\substack{q\le Q \\ \fa_1\mid (q)}} \frac{1}{q} \ll (\log Q) (N\fa_1)^{-\frac1n}.$$ Using $\sum^*$ to denote a sum over only primitive characters, \eqref{equation:reduction-to-non-trivial-zeros} is

\begin{equation} \label{equation:reduction-to-primitive-characters}
    \begin{split}
        &\ll (\log Q) \sum_{N\fa\le Q^n} ~\sideset{}{^*}\sum_{\mu \bmod \fa} (N\fa)^{-\frac1n} \sum_{\substack{\mathbf{m}\in \ZZ^d \\ \norm{\mathbf{m}}\le M}}  \sum_{L(\rho ,\mu \lam^\mathbf{m}) = 0} \frac{\abs{G_u^{\pm}(\rho)}}{h}
        \\&\ll (\log Q)^2 \max_{1\le R\le Q^n}R^{-\frac1n} \sum_{N\fa\le R} ~\sideset{}{^*}\sum_{\mu \bmod \fa}  \sum_{\substack{\mathbf{m}\in \ZZ^d \\ \norm{\mathbf{m}}\le M}}  \sum_{L(\rho ,\mu \lam^\mathbf{m}) = 0} \frac{\abs{G_u^{\pm}(\rho)}}{h},
    \end{split}
\end{equation}
where the last step uses a standard dyadic decomposition technique that decomposes the interval $[1, Q]$ into sub-intervals of the form $[2^{k}, 2^{k+1}),\, k=0,1,\ldots,\ceil{\log_2 Q}$.

\subsection{Zero-density estimates}
We prove that for $1\le R\le Q^n$ we have $$U(R) = R^{-\frac1n} \sum_{N\fa\le R} ~\sideset{}{^*}\sum_{\mu \bmod \fa}  \sum_{\substack{\mathbf{m}\in \ZZ^d \\ \norm{\mathbf{m}}\le M}}  \sum_{L(\rho ,\mu \lam^\mathbf{m}) = 0} \frac{\abs{G_u^{\pm}(\rho)}}{h}\ll_A (\log x)^{-A}.$$ This will conclude the proof of Theorem \ref{theorem:bombieri-vinogradov-galois-primes} by \eqref{equation:sum-over-zeros-inequality} and \eqref{equation:reduction-to-primitive-characters}. We need the following results concerning the zeros of Hecke $L$-functions.

\begin{lemma}[\cite{fogels}] 
\label{lemma:siegel-theorem-hecke}
For any $\eps > 0$, there is a constant $c_1(\eps) > 0$ such that for any real finite-order Hecke character $\mu\bmod \ff$ we have $L(\sigma, \mu) \neq 0$ for $\sigma \geq 1 - c_1(\eps)(N\ff)^{-\eps}$.
\end{lemma}

\begin{lemma}[{\cite[Theorem A]{jeffrey-ramakrishnan}}]
\label{lemma:siegel-zero-sparsity-of-modulus}
There is a constant $c_2 > 0$ such that there is at most one character $\mu\lam^\mathbf{m}$ among the primitive characters with conductor norm at most $R$ whose $L$-function $L(s,\mu\lam^\mathbf{m})$ has at a real zero in the region $\Re s \geq 1 - \frac{c_2}{\log R}$, known as an exceptional zero. For this character, one must have $\mathbf{m} = \mathbf{0}$ and $\mu$ real. Furthermore, there can be at most one exceptional zero for this character, and it must be simple.
\end{lemma}

\begin{lemma} \label{lemma:duke-zero-density-estimate}
There is a constant $B > 0$ such that for all $\sigma \in \big[\frac{1}{2}, 1\big)$, $R > 1$ and $M > 2$,
$$\sum_{N\fa \leq R}~\sideset{}{^*}\sum_{\mu \bmod \fa} \sum_{\substack{\mathbf{m}\in \ZZ^d \\ \norm{\mathbf{m}}\le M}} N(\sigma, M,\mu\lam^{\mathbf{m}}) \ll (R^2M^n)^{\min\bigp{\frac{3}{2-\sigma},\frac{2}{\sigma}}(1-\sigma)} (\log RM)^B$$
\end{lemma}
\begin{proof}
The zero-density estimate
\[
\sum_{N\fa \leq R}~\sideset{}{^*}\sum_{\mu \bmod \fa} \sum_{\substack{\mathbf{m}\in \ZZ^d \\ \norm{\mathbf{m}}\le M}} N(\sigma, M,\mu\lam^{\mathbf{m}}) \ll (R^2M^n)^{\frac{3}{2-\sigma}(1 - \sigma)} (\log RM)^B
\]
follows from work of Duke, who proved the corresponding result when $R=1$ is bounded \cite[Theorem 2.1]{duke}.  The only notable difference is that instead of applying the large sieve in \cite[Theorem 1.1 (i)]{duke}, we apply the large sieve in \cite[Theorem 1.1 (ii)]{duke}. The fourth-moment bound in \cite[Section 2.1]{duke}  and the zero detection process in \cite[Section 2.2]{duke} carry over completely analogously.

Duke's proof extends Montgomery's zero detection method for Dirichlet $L$-functions \cite{montgomery} to $L$-functions of the form $L(s,\mu\lambda^{\mathbf{m}})$.  As detailed in \cite{montgomery}, the method can be refined using results on the frequency with which certain Dirichlet polynomials attain large values \cite{montgomery-mean}.  Montgomery's work on large values of Dirichlet polynomials was extended to averages of $L$-functions of the form $L(s,\mu)$ by Hinz \cite{hinz}. Coleman \cite{coleman-general} outlined how Montgomery's work extends to averages of $L$-functions of the form $L(s,\mu\lambda^{\mathbf{m}})$. This results in the claimed improvement in the exponent.
\end{proof}

Let $\beta$ denote the exceptional zero for each $\mu \lambda^\mathbf{m}$, if such a zero exists. Then $U(R)$ can be split into the three terms

$$U_1(R) = R^{-\frac1n} \sum_{N\fa\le R} ~\sideset{}{^*}\sum_{\mu \bmod \fa}  \sum_{\substack{\mathbf{m}\in \ZZ^d \\ \norm{\mathbf{m}}\le M}}  \sum_{\substack{L(\rho ,\mu \lam^\mathbf{m}) = 0 \\ \rho = \beta}} \frac{\abs{G_u^{\pm}(\rho)}}{h}$$

$$U_2(R) = R^{-\frac1n} \sum_{N\fa\le R} ~\sideset{}{^*}\sum_{\mu \bmod \fa}  \sum_{\substack{\mathbf{m}\in \ZZ^d \\ \norm{\mathbf{m}}\le M}}  \sum_{\substack{L(\rho ,\mu \lam^\mathbf{m}) = 0 \\ \abs{\Im \rho} > M \\ \rho\neq \beta}} \frac{\abs{G_u^{\pm}(\rho)}}{h}$$

$$U_3(R) = R^{-\frac1n} \sum_{N\fa\le R} ~\sideset{}{^*}\sum_{\mu \bmod \fa}  \sum_{\substack{\mathbf{m}\in \ZZ^d \\ \norm{\mathbf{m}}\le M}}  \sum_{\substack{L(\rho ,\mu \lam^\mathbf{m}) = 0 \\ \abs{\Im \rho} \le M \\ \rho\neq \beta}} \frac{\abs{G_u^{\pm}(\rho)}}{h}.$$

We first bound $U_1(R)$. By Lemma \ref{lemma:siegel-zero-sparsity-of-modulus}, $U_1(R)$ contains at most one summand corresponding to the exceptional zero. If $\beta$ exists, then by Lemma \ref{lemma:siegel-theorem-hecke} with $\eps = \frac{1}{2nA}$ and \eqref{equation:mellin-transform-decay} with $\ell = 0$ we have
$$U_1(R) \leq R^{-\frac1n} x^{\beta - 1} \le R^{-\frac1n} x^{\frac{c_1(\eps)}{R^\eps}} \ll_A (\log x)^{- A},$$ where the last step follows by splitting into the two cases $R>e^{(\log x)^{\frac12}}$ and $R\le e^{(\log x)^{\frac12}}$.

Next, we bound $U_2(R)$. Applying partial summation, Lemma \ref{lemma:duke-zero-density-estimate}, and \eqref{equation:mellin-transform-decay}, we have
\begin{align*}
   U_2(R) &\ll_\ell R^{-\frac1n} \sum_{N\fa\le R} ~\sideset{}{^*}\sum_{\mu \bmod \fa}  \sum_{\substack{\mathbf{m}\in \ZZ^d \\ \norm{\mathbf{m}}\le M}} u^{-\ell} x^{\ell}\int_{M}^\infty t^{-\ell}\,dN\bigp{\frac{1}{2}, t, \mu\lam^{\mathbf{m}}}\,dt
   \\&\ll_\ell R^{-\frac1n} \sum_{N\fa\le R} ~\sideset{}{^*}\sum_{\mu \bmod \fa}  \sum_{\substack{\mathbf{m}\in \ZZ^d \\ \norm{\mathbf{m}}\le M}} u^{-\ell} x^{\ell}\bigp{ N\bigp{\frac{1}{2}, M, \mu\lam^{\mathbf{m}}} M^{-\ell} + \int_{M}^\infty N\bigp{\frac{1}{2}, t, \mu\lam^{\mathbf{m}}}\ell t^{-\ell - 1}\, dt}
   \\&\ll_\ell R^{-\frac1n}  u^{-\ell} x^{\ell}\bigp{ R^2 M^{n-\ell} (\log QM)^B + \int_{M}^\infty  R^2 \ell t^{n-\ell-1} (\log Qt)^B \,dt}
   \\&\ll_\ell R^{2-\frac1n} M^n u^{-\ell} M^{-\ell} x^{\ell} (\log x)^B.
\end{align*}
Since $\tau > \tau'$, we conclude that for $\ell$ sufficiently large we have $$U_2(R)\ll Q^{2n-1} M^n x^{- (\tau - \tau')\ell}(\log x)^B \ll_A (\log x)^{-A}.$$

Finally, we bound $U_3(R)$. By Theorem \ref{theorem:coleman-zero-free-region} noting that $M = x^{\tau}$, all non-trivial zeros $\rho \neq \beta$ with $\abs{\Im \rho} \le M$ satisfy $$\Re \rho \le 1 - \frac{c_3}{\max\bigp{\log R, (\log x)^{\frac23} (\log \log x)^{\frac13}}}$$ for some $c_3 > 0$. Denoting $$\sigma^* = 1 - \frac{c_3}{\max\bigp{\log R, (\log x)^{\frac23} (\log \log x)^{\frac13}}},$$ we conclude that
\begin{align*}
    U_3(R) &\ll R^{-\frac1n} \sum_{N\fa\le R} ~\sideset{}{^*}\sum_{\mu \bmod \fa}  \sum_{\substack{\mathbf{m}\in \ZZ^d \\ \norm{\mathbf{m}}\le M}}  \sum_{\substack{L(\rho ,\mu \lam^\mathbf{m}) = 0 \\ \abs{\Im \rho} \le M \\ \rho\neq \beta}} x^{\Re\rho -1}
    \\&\ll R^{-\frac1n} \max_{\sigma \in \bigb{\frac12, \sigma^*}} \sum_{N\fa\le R} ~\sideset{}{^*}\sum_{\mu \bmod \fa}  \sum_{\substack{\mathbf{m}\in \ZZ^d \\ \norm{\mathbf{m}}\le M}} 
    N(\sigma, M, \mu\lam^{\mathbf{m}})x^{\sigma - 1}\log x
    \\&\ll R^{-\frac1n} \max_{\sigma \in \bigb{\frac12, \sigma^*}}(R^2 M^n)^{\min\bigp{\frac{3}{2-\sigma}, \frac{2}{\sigma}}(1-\sigma)}x^{\sigma - 1}(\log x)^{B+1}
    \\&\ll R^{-\frac1n} \max_{\sigma \in \bigb{\frac12, \sigma^*}}(R^2 M^n)^{\frac{5}{2}(1-\sigma)}x^{\sigma - 1}(\log x)^{B+1}.
\end{align*}
Recalling that $M = x^\tau$ and $R\le Q^n = x^{\theta n}$, we have
\begin{align*}
 U_3(R) \ll& R^{-\frac1n} \max_{\sigma \in \bigb{\frac12, \sigma^*}}(x^{2\theta n} x^{\tau n})^{\frac{5}{2}(1 - \sigma)}x^{\sigma - 1} (\log x)^{B+1} \\
 \ll& R^{-\frac1n}\max_{\sigma \in \bigb{\frac12, \sigma^*}}x^{-\bigp{1 - \frac{5n}{2}(2\theta + \tau)}(1-\sigma)} (\log x)^{B+1}
\end{align*}
If $R > e^{(\log x)^{\frac{1}{2}}}$ then $R^{-1/n}\ll_A (\log x)^{-A-B-2}$, so we are done. If $R \leq e^{(\log x)^{\frac{1}{2}}}$, then $$1 - \sigma^* \gg (\log x)^{-\frac{2}{3}} (\log \log x)^{-\frac{1}{3}},$$ so $$x^{-\bigp{1 - \frac{5n}{2}(2\theta + \tau)}(1-\sigma)} \ll_A  (\log x)^{-A-B-2}.$$ We conclude the uniform estimate $U_3(R) \ll_A (\log x)^{-A-1}$ and thus Theorem \ref{theorem:bombieri-vinogradov-galois-primes}.

\end{document}